\documentclass[11pt,openany,leqno]{article}
\usepackage{amsmath,amsthm,amsfonts,amssymb,amscd,url}
\usepackage{graphics}
\usepackage[latin1]{inputenc}

\usepackage{enumerate}
\begin{document}
\def\Diff{\text{Diff}}
\def\Max{\text{max}}
\def\Log{\text{log}}
\def\loc{\text{loc}}
\def\inta{\text{int }}
\def\det{\text{det}}
\def\exp{\text{exp}}
\def\Re{\text{Re}}
\def\lip{\text{Lip}}
\def\leb{\text{Leb}}
\def\dom{\text{Dom}}
\def\diam{\text{diam}\:}
\newcommand{\ovfork}{{\overline{\pitchfork}}}
\newcommand{\ovforki}{{\overline{\pitchfork}_{I}}}
\newcommand{\Tfork}{{\cap\!\!\!\!^\mathrm{T}}}
\newcommand{\whforki}{{\widehat{\pitchfork}_{I}}}
\theoremstyle{plain}
\newtheorem{theo}{\bf Theorem}[section]
\newtheorem{lemm}[theo]{\bf Lemma}
\newtheorem{sublemm}[theo]{\bf Sublemma}
\newtheorem{IH}[theo]{\bf Extra induction hypothesis}
\newtheorem{prop}[theo]{\bf Proposition}%[section]
\newtheorem{conj}[theo]{\bf Conjecture}%[section]

\newtheorem{coro}[theo]{\bf Corollary}%[section]
\newtheorem{Property}[theo]{\bf Property}%[section]
\newtheorem{Claim}[theo]{\bf Claim}
\newtheorem{Assertion}[theo]{\bf Assertion}

\theoremstyle{remark}
\newtheorem{rema}[theo]{\bf Remark}
\newtheorem{remas}[theo]{\bf Remarks}

\newtheorem{exem}[theo]{\bf Example}
\newtheorem{Examples}[theo]{\bf Examples}
\newtheorem{defi}[theo]{\bf Definition}
\title{On the inverse limit stability of endomorphisms}

\author{Pierre Berger,\\
\small{LAGA Institut Galilée, Université Paris 13, France}\\
Alvaro Rovella,\\\small{Facultad de Ciencias, Uruguay}}
\date{may 2010}

\maketitle
\begin{abstract}
We present several results suggesting that the concept of $C^1$-inverse limit stability is free of singularity theory.  We describe an example of a $C^1$-inverse stable endomorphism which is robustly transitive with persistent critical set. We show that every (weak) axiom A, $C^1$-inverse limit stable endomorphism satisfies a certain strong transversality condition $(T)$. We prove that every attractor-repellor endomorphism satisfying axiom A and Condition $(T)$ is $C^1$-inverse limit stable. The latter is applied to Hénon maps, rational functions of the sphere and others. This leads us to  conjecture that $C^1$-inverse stable endomorphisms are those which satisfy axiom A and the strong transversality condition $(T)$.
\end{abstract}

\section{Introduction}
There exists various concepts of stability for dynamical systems. When dealing with endomorphisms it makes sense to consider the inverse limit which is defined in the sequel. A \emph{$C^1$-endomorphism $f$} is a $C^1$-map of a manifold $M$ into itself, which is not necessarily bijective and which can have a nonempty \emph{singular set} (formed by the points $x$ s.t. the derivative $T_x f$ is not surjective).
The \emph{inverse limit set} of $f$ is the space of the full orbits $(x_i)_i \in M^\mathbb Z$ of $f$. The dynamics induced by $f$ on its inverse limit set is the shift. The endomorphism $f$ is \emph{$C^1$-inverse limit stable} if for every $C^1$ perturbation $f'$ of $f$, the inverse limit set of $f'$ is homeomorphic to the one of $f$ via a homeomorphism which conjugates both induced dynamics and close to the canonical inclusion.

When the dynamics $f$ is a diffeomorphism, the inverse limit set $M_f$ is homeomorphic to the manifold $M$. The $C^1$-inverse limit stability of $f$ is
then equivalent to the $C^1$-\emph{structural stability} of $f$: every $C^1$-perturbation of $f$ is conjugated to $f$ via a homeomorphism of $M$.

 A great work was done by many authors to provide a satisfactory description of $C^1$-structurally stable diffeomorphisms, which starts with Anosov, Smale, Palis \cite{PS}, de Melo, Robbin, and finishes  with Robinson \cite{Rs} and Ma\~{n}\'e \cite{Mane}. Such diffeomorphisms are those which satisfy axiom A and the strong transversality condition.

 Almost the same description was accomplished for $C^1$-structurally stable flows by Robinson and Hayashi. The inverse limit set of a flow is a one dimensional foliation. The structural stability of a flow is also equivalent to the $C^1$-inverse  stability. A flow $\phi$ is \emph{structurally stable} if the foliation induced by $\phi$ is equivalent to the foliation induced by its perturbation, via a homeomorphism of $M$ which is $C^0$-close to the identity.

The descriptions of the structurally stable maps for smoother topologies ($C^r$, $C^\infty$, holomorphic...) remain some of the hardest, fundamental, open questions in dynamics.

One of the difficulties occurring in the description of $C^r$-structurally stable smooth endomorphisms concerns the singularities. Indeed, a structurally stable map must display a stable singular set. But there is no satisfactory description of them in singularity theory.

This work suggests that the concept of inverse limit stability does not deal with singularity theory.

The concept of inverse limit stability is an area of great interest for semi-flows given by PDEs, although still at its infancy.
\newline

{\small
The work of the first author was done during stays at IHES (France), IMPA (Brasil) and Facultad de Ciencias (Uruguay). He is very grateful to these institutes for their hospitality.}

\section{Statement of the main results}

Let $f$ be a $C^1$-map of a compact manifold $M$ into itself.

The \emph{inverse limit} of $f$ is the set $M_F:= \{\underline x\in M^\mathbb Z :\; f(x_i)=x_{i+1}\;\forall i\in \mathbb Z\}$, where $M^\mathbb Z$ is the space of sequences $\underline x=(x_i)_{i\in\mathbb Z}$. The subset $M_F$ endowed with the induced product topology is compact. The map $f$ induces the shift map $F(\underline x)_i= x_{i+1}$.
We remark that $M_F$ is equal to $M$ and $F$ is equal to $f$ if $f$ is bijective. The \emph{global attractor} of $f$ is defined as $M_f=\cap_{n\ge 0} f^n(M)$.
For $j\in \mathbb Z$, let:
\[
\pi_j:\; \underline x\in M_F\mapsto  x_j\in M_f.
\]
 We note that:
\[
\pi_j\circ F= f\circ \pi_j.
\]
Also a point $z$ belongs to $M_f$ if and only if $\pi_0^{-1}(\{z\})$ is not empty.
Although $\pi_j$ depends on $f$, this will be not emphasized by an explicit notation.

Two endomorphisms $f$ and $f'$ are $C^1$-\emph{inverse limit conjugated}, if there exists a homeomorphism $h$ from $M_F$ onto $M_{F'}$, such that the following equality holds:
\[
h\circ F= F'\circ h.
\]
\begin{defi}
An endomorphism $f$ is $C^1$-\emph{inverse limit stable} or simply $C^1$ \emph{inverse stable} if every $C^1$-perturbation $f'$ of $f$ is inverse limit conjugated to $f$ via a homeomorphism $h$ which is $C^0$ close to the inclusion $M_F\hookrightarrow M^\mathbb Z$.
\end{defi}

Let $K_f$ be a compact, $f$-invariant subset of $M$ ($f(K_f)\subset K_f$). Then $K_f$ is \emph{hyperbolic} if there exists a section $E^s$ of the Grassmannian of $TM|K_f$ and $N>0$  satisfying for every $x\in K_f$:
\begin{itemize}
\item $T_{x}f(E^s(x))\subset E^s(f(x))$,
\item the action $[Tf]$ induced by $f$ on the quotients $T_x M/E^s(x)\rightarrow T_{f(x)} M/E^s(f(x))$ is invertible,
\item $\|T_{x} f^N|E^s(x)\|<1$,
\item $\|([Tf]^{N})^{-1}\|<1$.
\end{itemize}

We notice that actually $E^s(\underline{x})$ depends only on $x_0=\pi_0(\underline x)$. It can be denoted by $E^s(x_0)$.

 On the other hand,
there exists a unique continuous family $\big(E^u(\underline x)\big)_{\underline x}$ of subspaces $E^u(\underline x)\subset T_{x_0}M$, indexed by $\underline x\in K_F:=K_f^\mathbb Z\cap M_F$, satisfying:
 
\[
T_{x_0}f(E^u(\underline x))=E^u(F(\underline x))
\quad
\mathrm{and}\quad
E^u(\underline x)\oplus E^s(x_0)=T_{x_0}M.
\]

For $\epsilon>0$, the $\epsilon$-local stable set of $x\in K_f$ is:

\[
W^s_\epsilon(x;f)=\big\{y \in M:\; \forall i\ge 0,\;d(f^i(x),f^i(y))\le \epsilon,\;\mathrm{ and}\; d(f^i(x),f^i(y))\rightarrow 0,\; i\rightarrow +\infty\big\}.
\]

The $\epsilon$-local unstable set of $\underline x\in K_F$ is:

\[W^u_\epsilon(\underline{x};F)=\big\{\underline y\in M_F:\; \forall i\le 0,\;d(x_i,y_i)\le \epsilon,\;\mathrm{ and}\;
d(x_i,y_i)\rightarrow 0,\; i\rightarrow -\infty\big\}.\]

Let us justify why we have chosen $W^s_\epsilon(x)$ included in $M$ whereas $W^u_\epsilon(\underline{x})$ is included in $M_F$.
One can prove that (for $\epsilon$ small enough) the local stable set is a submanifold whose tangent space at $x$ equals $E^s(x_0)$; however its preimage $W^s_\epsilon(\underline x; f)$ by $\pi_0$ is in general not a manifold (not even a lamination in general).
%This manifold depends actually only on $x_0$. This is why sometime we will denote it by $W^s_\epsilon(x_0)$.
The local unstable set is a manifold embedded into $M$ by $\pi_0$; its tangent space at $x_0$ is equal to $E^u(\underline x)$. In general the unstable manifold depends on the preorbit: the unstable sets of different orbits in $\pi^{-1}(x_0)$ are not necessarily equal.

An endomorphism satisfies \emph{(weak) axiom A} if the nonwandering set $\Omega_f$ of $f$ is hyperbolic and equal to the closure of the set of periodic points.

In this work, we do not deal with \emph{strong axiom A} endomorphisms which satisfy moreover that the action on each of the basic pieces of $\Omega_f$ is either expanding or injective. This stronger definition is relevant for structural stability \cite{Przy}, but it is conjectured below to be irrelevant for inverse stability.

We put $\Omega_F:= \Omega_f^\mathbb Z\cap M_F$. Actually if the $f$-periodic points are dense in $\Omega_f$ then the $F$-periodic points are dense in $\Omega_F$. For the sets of the form $\pi_N^{-1}(B(x,\epsilon))\cap \Omega_F$, with $x\in \Omega_f$, $\epsilon>0$ and $N\in \mathbb Z$, are elementary open sets of $\Omega_F$ and contain periodic points.

Also if  $\Omega_f$ is hyperbolic the restriction of $F$ to $\Omega_F$ is expansive. For the $\epsilon$ unstable manifold $W^u_\epsilon(\underline x)$ intersects $W^s_\epsilon(\underline x)$ at the unique point $\underline x$ since $\pi_0$ restricted to $W^u_\epsilon(\underline x)$ is a homeomorphism and $\pi_0(W^u_\epsilon(\underline x))$ intersects $W^s_\epsilon(\pi_0(\underline x))$ at the unique point $\pi_0(\underline x)$, for every $\underline x\in \Omega_F$.

\begin{defi} The dynamics $f$ satisfies the \emph{strong transversality condition} if:

For all  $n\ge 0$, $\underline x\in \Omega_F$ and $y\in \Omega_f$, the map $f^n$ restricted to $\pi_0\big(W^u_\epsilon (\underline x;F)\big)$ is transverse to $W^s_\epsilon(y;f)$. In other words, for every $z\in \pi_0\big(W^u_\epsilon (\underline x)\big)\cap f^{-n}\big(W^s_\epsilon(y)\big)$:
\begin{itemize}
\item[$(T)$]
\qquad \;\qquad$
T_zf^n\Big(T_z \pi_0 \big(W^u_\epsilon (\underline x)\big)\Big)+T_{f^{n}(z)} W^s_\epsilon (y)=T_{f^{n}(z)}M.$
\end{itemize}
\end{defi}
A first result is:
\begin{theo}
\label{th1}
Let $M$ be a compact manifold and $f\in C^1(M,M)$. If $f$ is $C^1$-inverse stable and satisfies axiom A, then the strong transversality condition holds for $f$.
\end{theo}

The second one concerns the converse:
\begin{defi}
\label{AR}
An axiom A endomorphism is \emph{attractor-repeller} if $\Omega_f$ is the union of two subsets $R_f$ and $A_f$ such that there exist:
\begin{itemize}
\item a neighborhood $V_A$ of $A_f$ \emph{in $M$} satisfying $\bigcap_{n\ge 0} f^{n}(V_A)=A_f$,
\item a neighborhood $V_R$ of $R_f$ \emph{in $M_f$} satisfying $\bigcap_{n\ge 0} f^{-n}(V_R)=R_f$.
\end{itemize}
The set $R_f$ is called a \emph{repeller} and $A_f$ an \emph{attractor}.
\end{defi}

\begin{theo}
\label{th2}
Let $M$ be a compact manifold and $f\in C^1(M,M)$. If $f$ is an attractor-repeller endomorphism which satisfies the strong transversality condition, then $f$ is $C^1$-inverse stable.
\end{theo}

It follows immediately from the Theorem of Aoki-Moriyasu-Sumi in \cite{AMS} that:

If an endomorphism $f$ is $C^1$-inverse stable and has no singularities in the nonwandering set, then $f$ satisfies axiom A.

Hoping to generalize this result and Theorem \ref{th2}, we propose the
following  conjecture (vaguely written in \cite{Quandt-ano}):
\begin{conj}
The $C^1$-inverse stable endomorphisms are exactly those which satisfy axiom A and the strong transversality condition.
\end{conj}

\subsection{Application of Theorem \ref{th1}}

\begin{exem}[Rational functions] Let $f$ be a rational function of the Riemann sphere. Let us suppose that all its critical points belong to basins of attracting periodic orbits, or equivalently that its Julia set is expanding. By Theorem \ref{th2}, $f$ is $C^1$-inverse stable.
Note that $C^1$-perturbations of $f$ may have very wild critical set.
See \cite{Lyu} for a nice geometrical description of the inverse limit of $f$.\end{exem}

\begin{exem}[One-dimensional dynamics and Henon maps]\label{onedim}
Kozlovski-Shen-van Strien showed that a $(C^\infty)$-generic map $f$ of the circle $\mathbb S^1$ is attractor-repeller  (\cite{Shen}), and so  $C^1$-inverse limit stable, by Theorem \ref{th2}.

Let $f'(\theta,y)=(f(\theta)+y,0)$ be defined on the 2-torus $\mathbb T^2$ which enjoys of a canonical Abelian group structure.
%
%cylinder $S^1\times \mathbb R$, where $S^1$ is seen as the quotient $\mathbb R/\mathbb Z$.
%
Aside finitely many attracting periodic points, the nonwandering set of $f'$ consists of an expanding compact set of $f$ times $\{0\}$. This product $R$ is a hyperbolic set for $f'$ and a repeller (for the restriction of $f'$ to $M_{f'}$), as stated in definition \ref{AR}. It follows that $f'$ satisfies the requirements of Theorem \ref{th2}. This implies that if $g\in C^1(\mathbb T^2, \mathbb R)$ is close to $0$, then the inverse limits sets of $f$ and of the map:
\[
(\theta, y)\mapsto\big(f(\theta)+y,g(\theta,y)\big), 
\]
are conjugated.

For instance, take $f(x)=x^2+c$ with $c\in (-2,1/4)$ attractor-repellor on the one-point compactification of $\mathbb R$. The infinity is an attracting fixed point with basin bounded by the positive fixed point $p$ of $f$ and its preimage. Let $\rho$ be a smooth function with compact support in $\mathbb R$ and equal to 1 on a neighborhood of $[-p,p]$.

 We get that for $b$ small enough, the attractor of the Hénon map $(x,y)\mapsto (x^2+c+y, bx)$ of $\mathbb R^2$, equals to the one of $(x,y)\mapsto (x^2+c+y, \rho(x) \cdot b\cdot x)$ without the basin of $(\infty,0)$, is conjugated to the inverse limit of $f|[-p,p]$. 
 
 The same example works with $f$ a hyperbolic rational function of the sphere. This generalizes many results in this direction to the wide $C^1$-topology (see \cite{Hub} which contains other references).
 \end{exem}

\begin{exem}[Anosov endomorphisms with persistent critical set]
Przytycki  showed that an Anosov endomorphism without singularities is inverse stable \cite{Przy}. Latter Quandt  generalized this for Anosov endomorphisms, possibly with singularities \cite{Quandt-ano}. These results are consequences of Theorem \ref{th2}.

The simplest known example of Anosov endomorphisms are action of linear maps on the quotient $\mathbb R^2/\mathbb Z^2$, for instance:
\[A=\left[ \begin{array}{cc}n&1\\
1& 1\end{array}
\right],\quad n\in\{2,3,\dots\}.\]

A constant map is a trivial example of an Anosov endomorphism. Let us construct an example of Anosov map whose singular set is persistently nonempty and whose nonwandering set is the whole manifold.

Begin with a linear map $A$ of the plane as above. Close to the fixed point one can use linear coordinates to write the map as

\[\left[ \begin{array}{cc}\lambda&0\\
0& \mu\end{array}
\right],\]
where $0<\mu<1$ and $\lambda>1$. Let $\epsilon$ be a positive constant and let $\Psi$ be a nonnegative smooth function such that $\Psi(0)=1$ and $\Psi(x)=0$ for every $|x|>\epsilon$. Assume also that $\Psi$ is an even function having a unique critical point in $(-\epsilon,\epsilon)$. Let $\varphi$ be the $C^1$ function defined by:
$\varphi(y)=0$ for every $y\notin [0,\epsilon]$ and $\varphi'(y)=\sin (\frac{2\pi y}{\epsilon})$ for $y\in [0,\epsilon]$. Let $f$ be the $C^1$-endomorphism of the torus equal to
\[f(x,y)=\big(\lambda x,\mu y-\Psi(x)\varphi(y)\big)\]
on the $2\epsilon$-neighborhood of $0$ and to $A$ off.

Let $g$ be the real function $y\mapsto \mu y-\phi(y)$. There are regular points with different numbers of $g$-preimages. The same occurs for $f$. Consequently $f$ has a singular set which is persistently nonempty.%

We remark that $f$ is Anosov. For the $A$-stable direction is still preserved and contracted; the action of $Tf$ on the stable foliation normal bundle is still $\lambda$-expanding.

Moreover the stable leaves are irrational lines of the torus. From this it comes that given nonempty open sets $U$ and $V$,  $f^{-k}(U)$ contains a sufficiently long segment of such lines to intersect $V$ for every $k$ large. In other words, $f$ is mixing.
\end{exem}

\begin{exem}[Products]
Theorem \ref{th2} shows also the inverse stability of product of an Anosov endomorphism with an attractor-repeller endomorphism.
\end{exem}

\begin{exem}[Ma\~{n}\'e-Pugh \cite{MP}]
Ma\~{n}\'e and Pugh gave an example of $C^1$-$\Omega$-stable endomorphism for which the singular set persistently intersects an attracting basic piece. Their example is clearly not $C^r$-structurally stable but according to Theorem \ref{th2} it is $C^1$-inverse-limit stable
\end{exem}

\section{Proof of Theorem $\ref{th1}$}

We begin this section with two well known facts of transversality theory.

\begin{Claim}\label{aclaim}Let $N_1$ and $N_2$ be two embedded submanifolds of $M$.
\begin{enumerate}
\item[$(i)$]
The set of maps $f\in C^1(M,M)$ such that $f^n|N_1$ is transverse to $N_2$ for every $n\ge 1$ is residual.
\item[$(ii)$]
If $f$ is a $C^1$ map and $f{|N_1}$ is not transverse to $N_2$, then there exists a $C^1$ perturbation $f'$ of $f$ such that $f'^{-1}(N_2)\cap N_1$ contains a submanifold whose codimension is less than the sum of the codimensions of $N_1$ and $N_2$.

\end{enumerate}
\end{Claim}

Let $f$ be a $C^1$-inverse stable endomorphism satisfying axiom A. For each small perturbation $f'$ of $f$, let $h(f')$ be the conjugacy between $M_F$ and $M_{F'}$.

We will assume by contradiction that the transversality condition fails to be true. This means that there exist $n\geq 0$,
$\underline x\in \Omega_F$, $y\in \Omega_f$ and $z\in \pi_0\big(W^u_\epsilon(\underline x;F)\big)\cap f^{-n}\big(W^s_\epsilon (y;f)\big)$ such that Equation $(T)$ does not hold.
Note first that $z\notin \Omega_f$, by hyperbolicity of the nonwandering set.\\

 Moreover, by density of periodic orbits in $\Omega_F$ after replacing $f$ by a perturbation,
 we can assume that $\underline x$ and $y$ are periodic points.  To simplify the calculations, we can suppose that $\underline x$ and $y$ are fixed points by considering an iterate of $f$.

The conjugacy $h(f')$ was asked to be close to the inclusion of $M_F$ into $M^\mathbb Z$.
By expansiveness of $\Omega_F$, if a perturbation $f'$ is equal to $f$ at the nonwandering set $\Omega_f$, then $h(f')$ is equal to the inclusion of $\Omega_F$. We will produce perturbations $f'$ and $f''$ of $f$ that are equal to $f$ on  $\Omega_f$.

The second item of  Claim \ref{aclaim} can be used to produce a perturbation $f'$ of $f$ such that
\[
f'^{-n}\big(W^s_\epsilon(y;f')\big)\cap \pi_0\big(W^u_{\epsilon}(\underline x;F')\big)
\]
contains a submanifold of dimension $p>u+s-m$, where $u$ is the dimension of $\pi_0(W^u_\epsilon(\underline x))$, $s$ is the dimension of $W^s_\epsilon(x)$ and $m$ the dimension of the manifold $M$.

On the other hand, the first item of Claim \ref{aclaim} implies that for generic perturbations $f''$ of $f$, the restriction of $(f'')^k$ to $\pi_0(W^u_\epsilon(\underline x;F''))$ is transverse to $W^s_\epsilon(y;f'')$ for every positive integer $k$.

If $\epsilon>0$ is sufficiently small, the maps $f'$ and $f''$ are injective restricted to the closures of $\pi_0\big(W^u_\epsilon(\underline x;f')\big)$ and $\pi_0\big(W^u_\epsilon(\underline x;f'')\big)$ respectively. This implies that the restrictions of $\pi_0$ to the closures of $W^u_\epsilon(\underline x;f')$ and $W^u_\epsilon(\underline x;f'')$ are homeomorphisms onto their images.

For $\underline y\in \pi_0^{-1}(y)$, note that $\pi_0^{-1}(W^s(y;f'))$ and $\pi_0^{-1}(W^s(y;f''))$ are equal to the stable sets $W^s(\underline y;F')$ and $W^s(\underline y;F'')$ respectively.

Consequently $A':=W^s(\underline y;F')\cap W^u_\epsilon(\underline x;F')$ contains a manifold of dimension $p$ whereas
$A'':= W^u_{\epsilon}(\underline x;F'')\cap W^s(y;F''))$ is a (possibly disconnected) manifold of dimension $u+s-m<p$.

We assumed that $f$ is inverse stable, so the map $\phi:= h(f'')^{-1}\circ h(f')$ is a conjugacy between $F'$ and $F''$ which fixes $\underline y$ and $\underline x$. Thus $\phi$ must embed $A'$ into 
$W^u(\underline x; F'')\cap W^s(\underline y, F'')=\cup_{n\ge 0} F''^n(A'')$.

As $F''$ is homeomorphism, a manifold of dimension $p$, contained in $A'$, is embedded by $\phi$ into the manifold $\cup_{n\ge 0} F''^n(A'')$ of dimension less than $p$. This is a contradiction.

\section{Proof of Theorem $\ref{th2}$}

\subsection{General properties on axiom A endomorphisms}
Let us first remark that $\Omega_F:= \Omega_f^\mathbb Z\cap M_F$ is also the nonwandering set of $F$. Indeed, an elementary open set $U_F$ of $M_F$ has the form $(\prod_{i< N} M \times U\times \prod_{i>N} M)\cap M_F=\prod_{n\in \mathbb Z} f^{n-N}(U)$, where $U$ is an open set in $M$. Therefore $F^n(U_F)$ intersects $U_F$ for $n>0$, iff $f^n(U)$ intersects $U$.

The density of the periodic points in a compact hyperbolic set $K\subset M$ is useful to have a \emph{local product structure}, that is for every $\underline x, \underline y\in K_F=K^\mathbb Z\cap M_F$ close, the set $W^u_\epsilon(\underline x; F)$ intersects  $W^s_\epsilon(\underline x; F)$ at a unique point $[\underline x,\underline y]$ which \emph{belongs} to $K_F$.

 \begin{lemm} If the periodic points are dense in a compact hyperbolic set $K$, then $K_F$ has a local product structure.\end{lemm}
\begin{proof} If $\underline x, \underline y$ are close enough then $\pi_0 W^u_\epsilon(\underline x; F)$ intersects $\pi_0 W^s_\epsilon(\underline y;F)$ at a unique point $z$. Thus for every pair of periodic points $\underline x', \underline y'$ close to $\underline x, \underline y$, the local unstable manifold $\pi_0 W^u_\epsilon(\underline x'; F)$ intersects $\pi_0 W^s_\epsilon(\underline y';F)$ at a unique point $z'$. As $\pi_0 W^u_\epsilon(\underline y'; F)$ intersects $\pi_0 W^s_\epsilon(\underline x';F)$, the point $z'$ is nonwandering. Thus $z$ is nonwandering and also its preimage $[\underline x, \underline y]$ by $\pi_0|W^u_\epsilon(\underline x; F)$.\end{proof}

The existence of a local product structure is useful for the following:

\begin{lemm} A hyperbolic set equipped with a local product structure satisfies the shadowing property.\end{lemm}
\begin{proof}The proof of this lemma is treated as for diffeomorphisms.\end{proof}

\begin{lemm} If $f$ is attractor-repeller\footnote{Actually $f$ axiom A is sufficient.}, then $M_F$ is equal to the union of the unstable manifolds of $\Omega_F$'s points.\end{lemm}
\begin{proof}
Every point has its $\alpha$-limit set in $\Omega_F$, and so if the point $\underline x=(x_n)_n$ does not belong to $A_F$ then its $\alpha$-limit set is included in $\pi_0^{-1}(V_{R})$. Thus for $n<0$ small, the points $x_n$ remain close to $R_f$. By shadowing we show that such $x_n$ belong to local unstable manifolds of points in  $R_F:=R_f^\mathbb Z\cap M_F$.\end{proof}

If an attractor-repeller $f$ satisfies the strong transversality condition, then
for every $x\in \Omega_f$, and $n\ge 0$, the restriction  $f^n{|M_f}$ is transverse to $W^s_\epsilon(x)$. This means that for $z\in f^{-n}\big(W^s_\epsilon(x;f)\big)\cap M_f$, we have:
\begin{equation}
\label{stc2}
T_zf^n(T_z M)+T_{f^{n}(z)}\big(W^s_\epsilon(x;f)\big)=T_{f^{n}(z)}M.
\end{equation}

\subsection{Extension $\hat R_f$ of the repeller $R_f$}
Let us begin by showing some previous results:

\begin{lemm}\label{lem1} Let $f$ be an attractor-repeller endomorphism. Then $\hat R_f= W^s(R_f)\cap M_f=M_f\cap \cup_{n\ge 0} f^{-n}( R_f )$ is compact. Also if Equation (\ref{stc2}) holds, then $\hat R_f$ is hyperbolic and endowed with a product structure.
\end{lemm}

\begin{proof}
The set $\hat R_f$ is compact since it is the complement of the basin of $A_f$. For, every point $x$ has its $\omega$-limit set which is
either included in $R_f$ or in $A_f$. The first case
corresponds to $x$ in $\hat R_f$, the latter to $x$ in the basin of $A_f$.

We define on $\hat R_f$ the Grassmanian section $E^s$: $x\in f^{-n} (R_f)\cap M_f\mapsto (T_x f^n)^{-1} (E^s _{f^n (x)})$.

Let us show that $E^s$ is continuous. By $Tf$ invariance of this bundle, it is sufficient to show the continuity at $R_f$. At a point of $R_f$, the continuity follows from $(T)$ applied to $\underline x\in R_F$ and $y\in R_f$, and furthermore the lambda lemma.  

Let $E^u:= (TM|\hat R_f)/E^s$ and denote by $[Tf]$ the action of $Tf$ on this quotient bundle.

For every $x\in \hat R_f$ there exists $N\ge 0$ such that for every $n\ge N$:
\begin{itemize}
\item $T_x f^n |E^s$ is $\frac{1}{2}$-contracting,
\item $[T f^n](x)$ is $2$-expanding.
\end{itemize}
By compactness of $\hat R_f$, there exists $N>0$ such that
for all $x\in \hat R_f$:
\[\left\{\begin{array}{l}\|T_x f^n |E^s \|<1\\
\|[T f]^n (u)\|>1\end{array}\right.\]
 for every unit vector $u\in E^u$.

In other words $\hat R_f$ is hyperbolic. As $\hat R_F$ is equal to its stable set, it has a local product structure.
\end{proof}

Let us now suppose that $f$ is an attractor-repeller endomorphism which satisfies the strong transversality condition.

Let us denote by $\hat R_F:= M_F\cap \hat R_f^\mathbb Z$. We remark that $\hat R_f$ is $f|M_f$-invariant: $f^{-1}(\hat R_f)\cap M_f=\hat R_f$.

\begin{lemm}\label{unautrelemm} The set  $V_{\hat R_F}=W^u_\epsilon(\hat R_F):= \cup_{\underline x} W^u_\epsilon(\underline x;F)$ is a neighborhood of $\hat R_f$.\end{lemm} 
By shadowing this lemma is an easy consequence of the following:
\begin{sublemm} Every orbit $\underline x\in M_F$ such that $\pi_0(\underline x)$ is close to $\hat R_f$ satisfies that $\pi_{-n}(\underline x)$ is close to $\hat R_f$ for every $n\ge 0$. 
\end{sublemm}
\begin{proof} Otherwise there exists $\delta>0$ and a sequence of orbits $(\underline x^n)_n$ such that $(\pi_0(\underline x^n))_n$ approaches $\hat R_f$ but there exists $(m_n)_n$ such that $\pi_{-m_n} (\underline x^n)$ is $\delta$-distant to $\hat R_f$ for every $n$. Let $y$ be an accumulation point of $(\pi_{-m_n} (\underline x^n))_n$. The point $y$ cannot lie in the (open) basin of $A_f$ and so its $w$-limit set is included in $\hat R_f$. By shadowing, $y$ belongs to a local stable manifold of a  point $\hat R_f$. In other words $y$ belongs to $\hat R_f$, this is a contradiction.\end{proof}

%\rightarrow  
%Let $V_{\hat R_f}:= \cup_{n\ge 0} f^{-n}(V_R)$. We have $f^{-1}(V_{\hat R_f})\subset V_{\hat R_f}$ and $\cap_{n\ge 0} f^{-n}(V_{\hat R_f})= \hat R_f$.
%
%Let $V_{\hat R_F}:= \pi_0^{-1}(V_{\hat R_f})=M_F\cap \prod_{n<0} M\times V_{\hat R_f}\times \prod_{n>0} M$.
%
%We notice that   $V_{\hat R_F}$ is a neighborhood of $\hat R_F$ in $M_F$. Moreover each of its point has its $\alpha$-limit set included in $\hat R_F$. As $\hat R_F$ is equal to its stable set, it has a local product structure.
%
%Thus by shadowing, $V_{\hat R_F}$ is covered by local unstable manifolds of points in $\hat R_F$.
%This is why we can suppose that:
% \[V_{\hat R_F}=W^u_\epsilon(\hat R_F):= \cup_{\underline x} W^u_\epsilon(\underline x;F)\]

We fix $\hat \epsilon>\epsilon>0$ sufficiently small in order that  $\pi_0(W^u_{\hat \epsilon}(\underline x))$ is a submanifold embedded by $f$ for every $\underline x\in \hat R_f$.

\subsection{Stratification of $M_F$ by laminations}
As $f$ is in general not onto, we have to keep in mind that we work only on $M_F$. This set is in  general not a manifold not even a lamination (see for instance Example \ref{onedim}). However we are going to stratify it into three laminations suitable to construct the conjugacy. Let us recall some elements of the lamination theory applied to hyperbolic dynamical systems.

 A \emph{lamination} is a secondly countable metric space $L$ locally modeled on  open subsets $U_i$ of products of $\mathbb R^n$ with locally compact metric spaces $T_i$ (via homeomorphisms called \emph{charts}) such that the changes of coordinates  are of the form:
\[\phi_{ij}=\phi_j\circ \phi_i^{-1}:\; U_i\subset \mathbb R^n\times T_i\rightarrow U_j\subset \mathbb R^n\times T_j\]
\[(x,t)\mapsto (g(x,t),\psi(x,t)),\]
where the partial derivative w.r.t. $x$ of $g$ exists and is a continuous function of both $x$ and $t$, also $\psi$ is locally constant w.r.t. $x$.
A maximal atlas  $\mathcal L$ of compatible charts is a \emph{lamination structure} on $L$.

A \emph{plaque} is a component of $\phi_i^{-1}(\mathbb R^n\times \{t\})$ for a chart $\phi_i$ and $t\in T_i$.
The \emph{leaf} of $x\in L$ is the union of all the plaques which contain $x$. A leaf has a structure of manifold of dimension $n$. The \emph{tangent space} $T\mathcal L$ of $\mathcal L$ is the vector bundle over $L$  whose fiber $T_x \mathcal L$ at $x\in L$ is the tangent space at $x$ of its leaf.

The stratification is made by the two $0$-dimensional laminations (leaves are points) supported by $A_F$ and $R_F$, and by the lamination $\mathcal L$ supported by $M_F\setminus \Omega_F$ whose leaves are the intersections of stable and unstable manifolds components. The construction of $\mathcal L$ is delicate and is the object of this section.

We prefer to see $A_F$ and $R_F$ as laminations because  they turn out to be non trivial laminations for similar problem (semi-flow, bundle over attractor-repeller dynamics).

Let us construct a laminar structure on $W^u(\hat R_F):=\cup_{n\ge 0} F^n(W^u_\epsilon (\hat R_F))$.

\begin{prop}\label{3.3} The set $W^u(\hat R_F)$ is endowed with a structure of lamination $\mathcal L^u$ whose plaques are local unstable manifolds.
\end{prop}
 \begin{proof}
First we notice that $W^u(\hat R_F)$ is equal to the increasing open union $\cup_{n\ge 0} F^n(W^u_\epsilon (\hat R_F))$. As $F$ is a homeomorphism of $M_F$, we just need to exhibit a laminar structure on $W^u_\epsilon (\hat R_F)$.

Let us express some charts of neighborhoods of any $\underline x\in \hat R_F$ that span the laminar structure on $W^u_\epsilon (\hat R_F)$.
 For every $\underline y\in \hat R_F$ close to $\underline x$, the intersection of  $W^u_\epsilon (\underline y)$ with $W^s_\epsilon(\underline x)$ is a point $t=[\underline y,\underline x]$ in $\hat R_F$ by Lemma \ref{lem1}.  Also we can find a family of homeomorphisms $(\phi_t)_t$ which depends continuously on $t$ and sends $W^u _\epsilon(t)$ onto $\mathbb R^d$. We notice that the map:
\[ \underline y\mapsto (\phi_t(\underline y), t)\in \mathbb R^d \times W^s_\epsilon (\underline x)\]
 is a homeomorphism which is a chart of lamination.
 \end{proof}

It is well known that $W^s_\epsilon(A_f)$ has a structure of lamination, whose leaves are local stable manifolds (a direct proof is similar and simpler  than the one of Proposition \ref{3.3})

To construct the last lamination we are going to proceed by transversality. Let us recall some general definitions and facts.

We recall that a continuous map $g$ from a lamination $\mathcal L$ to a manifold $M$ is \emph{of class $C^1$} if its restriction to every plaque of $\mathcal L$ is a $C^1$ map of manifolds and the induced map $Tg:\; T\mathcal L\rightarrow TM$ is continuous. This means that the fiber restriction $T_xg: \; T_x \mathcal L\rightarrow T_{g(x)}M$ depends continuously on $x\in \mathcal L$.

For instance the restriction of $\pi_0$ to $\mathcal L^u$ is of class $C^1$. The \emph{tangent space} $T_x\mathcal L$ of the lamination at $x\in \mathcal L$ is the tangent space of the plaque at $x$.

Let $\mathcal L'$ be a lamination embedded into $M$. The lamination $\mathcal L$  is \emph{transverse to $\mathcal L'$ via $g$} if for every $x\in \mathcal L$ such that $g(x)$ belongs to $\mathcal L'$, the following inequality holds:
\[Tg(T_x \mathcal L)+ T_{g(x)}\mathcal L'=T_{g(x)}M\]

The concept of transversality is useful from the following fact:
\begin{Claim} There exists a lamination $\mathcal L\pitchfork_g \mathcal L'$ on $\mathcal L\cap g^{-1}(\mathcal L')$ whose plaques are intersections of $\mathcal L$-plaques with $g$-preimages of $\mathcal L'$-plaques.\end{Claim}
\begin{proof}
Let us construct a chart of $\mathcal L\pitchfork_g \mathcal L'$ for distinguished open sets which cover $\mathcal L\cap g^{-1}(\mathcal L')$.
Let $x\in \mathcal L\cap g^{-1}(\mathcal L')$, let $\phi:\; U\rightarrow \mathbb R^d \times T$ be a $\mathcal L$-chart of a neighborhood $U$ of $x$ and
let $\phi':\; U\rightarrow \mathbb R^{d'} \times T'$ be a $\mathcal L'$-chart of a neighborhood $U'$ of $g(x)$.

 For each $t\in T$, let $T'(t)$ be the set of $t'\in T'$ s.t. $g\circ \phi^{-1}(\mathbb R^{d}\times \{t\})$ intersects $\phi'^{-1}(\mathbb R^{d'}\times \{t'\})$. Let $P_{t,t'}$ be the $g$-pull back of this intersection. We notice that
$P_{t,t'}$ depends continuously on $(t,t')\in \sqcup_{t\in T} T'(t)$ as a $C^1$-manifold of $M$. By restricting $U$, $P_{t,t'}$ is diffeomorphic to $\mathbb R^{d+d'-n}$, via a map $\phi_{t,t'}:\; P_{t,t'}\rightarrow \mathbb R^{d+d'-n}$ which depends continuously on $t,t'$.
This provides a chart:
\[ U\cap g^{-1}(U')\rightarrow \mathbb R^{d+d'-n}\times \bigsqcup_{t\in T} T'(t)\]
\[x\in P_{t,t'}\rightarrow (\phi_{t,t'}(x), (t,t'))\]
\end{proof}

We remark that $M$ is a lamination formed by a single leaf, so we can use this claim with $g=f^n$, $\mathcal L:= W^s_\epsilon(A_f)$ and $\mathcal L'=M$.

Using transversality we would like to endow $W^s(A_f)=\cup_{n\ge 0}f^{-n}( W^s_\epsilon(A_f)) $ with a structure of lamination, however $f^n$  is not necessarily transverse to $W^s_\epsilon(A_f)$ off $M_f$.

Nevertheless, by (\ref{stc2}), transversality occurs at a neighborhood $U_n$ of $M_f\cap f^{-n}(W^s_\epsilon(A_f))$. This implies the existence of a structure of lamination $\mathcal L^s$ on $U\cap W^s(A_f)$, with $U=\cup_{n\ge 0} U_n$.

By $(T)$, the map $\pi_0$ sends $\mathcal L^u$ transversally to $\mathcal L^s$, since every $\underline x\in \mathcal L^u$ is equal to an iterate  $F^n(\underline y)$ with $\underline y\in W^u_\epsilon(\hat R_f)$ and  $f^n\circ \pi_0$ is equal to  $\pi_0\circ F^n$.

This enables us to define the lamination $\mathcal L_F:= \mathcal L^u \pitchfork_{\pi_0} \mathcal L^s$ supported by $W^u(\hat R_F)\cap W^s(A_F)=M_F\setminus \Omega_F$, and whose leaves are components of the intersections of stable and unstable manifolds.

As the laminations $\mathcal L^u$ and $\mathcal L^s$ are preserved by $F$ and $f$ respectively, it follows that the lamination $\mathcal L_F$ is preserved by $F$.

Therefore the space $M_F$ is stratified by the three following laminations:
\begin{itemize} \item the $0$-dimensional lamination $A_F$ (leaves are points),
\item the $0$-dimensional lamination $\hat R_F$,
\item the lamination $\mathcal L_F$ defined above.
\end{itemize}

\subsection{Conjugacy}
By Proposition 1 of \cite{Quandt-ano}, the hyperbolic continuity theorem holds for the inverse limit of hyperbolic sets. In particular 
\begin{coro}
For $f'$ $C^1$-close to $f$ there exists an embedding $h$ of $A_F \sqcup \hat R_F$ onto $A_{F'} \sqcup \hat R_{F'}\subset M_{F'}$, and such that  $F'\circ h= h\circ F|\hat R_F\sqcup A_F$. Also $h$ is close to the canonical inclusion of $A_F\sqcup \hat R_F$ in $M^\mathbb Z$.\end{coro}

We are going to extend the conjugacy $h$ to $M_F$.

First we need the following:
\begin{prop}\label{liminvpert}
For $f'$ $C^1$-close to $f$, we have: \[M_{F'}=W^u(\hat R_{F'})\cup A_{F'}.\]
\end{prop}\begin{proof}

Let $F':= \underline x\in M^\mathbb Z\mapsto \big(f'(x_i)\big)_i$.

As $\hat R_F$ (resp. $A_F$) contains all its stable (resp. unstable) manifolds, the same occurs for $\hat R_{F'}$ and $A_{F'}$.

Let $V_1$ and $V_2$ be small open neighborhoods  in $M^\mathbb Z$ of $\hat R_F$ and $A_F$ respectively. By Lemma \ref{unautrelemm} (applied for $f'$), they satisfy: 
\[\cap_{n\ge 0}   \hat F'^{n}(V_1)\subset W^u_\epsilon(\hat R_{F'})\quad \mathrm{ and} \quad \cap_{n\ge 0} \hat F'^n(V_2)=A_{F'}.\]

As the $\omega$-limit set is included in $R_f\sqcup A_f$, by compactness there exists $N$ large such that $M_F\subset \hat F^{-N}(V_1)\cup V_2$ and
$\hat F^N(M^\mathbb Z)\subset \hat F^{-N}(V_1)\cup V_2$.

Consequently $M_{F'}\subset \hat F'^N(M^\mathbb Z)\subset \hat F'^{-N}(V_1)\cup V_2$, for $f'$ close enough to $f$.

Using the $F'$ invariance of $M_{F'}$, the latter is included in $W^u(\hat R_{F'})\cup A_{F'}$. \end{proof}

For an adapted metric, the open subset $W^u_\epsilon(\hat R_F)$ of $M_F$ satisfies that:
\[cl\Big(F^{-1}\big(W^u_\epsilon(\hat R_F)\big)\Big)\subset W^u_\epsilon(\hat R_F)\; \mathrm{and} \; \cap_{n\ge 0} F^{-n} \big(W^u_\epsilon(\hat R_F)\big)=\hat R_F.\]

Let $D_F:= W^u_\epsilon(\hat R_F)\setminus F^{-1}\big(W^u_\epsilon(\hat R_F)\big)$.

%As there exists $V_{A_F}$ a neighborhood of $A_F$ such that $\cap_{n\ge 0} F^n(V_{A_F})=A_F$, we have:
%      \[\bigcup_{n\in \mathbb Z} F^{n}(D_F)=\mathcal L_F=M_F\setminus \Omega_F.\]

We notice that $\cup_{n\in \mathbb Z} F^{-n}(D_F)= W^u(\hat R_F)\setminus \cap_{n\ge 0} F^{-n}(W^u_\epsilon(\hat R_F))=W^u(\hat R_F)\setminus \hat R_F$. A domain with this last property is called a \emph{fundamental domain for $W^u(\hat R_F)$}.

Let $\partial^- D_F:= F^{-1}\big(cl(D_F)\setminus D_{F}\big)$.

In the last section we are going to prove the following
\begin{lemm} \label{unlemmedeconju}

For $f'$ sufficiently $C^1$ close to $f$, there exists a homeomorphism $h_\#$ from a small open neighborhood $V$ of $cl(D_F)$ into $\mathcal L_{F'}$ such that:
\begin{itemize}
\item[$(i)$] the map $\pi_0\circ h_\#$ is $C^1$-close to $\pi_0$,
\item[$(ii)$] for all $\underline x\in A_F$, $\underline y\in \hat R_F$, $\underline z\in W^s(\underline x;F)\cap W^u(\underline y;F)\cap D_F$, the point $h_\#(\underline z)$ belongs to  $W^s\big(h(\underline x);F'\big)\cap W^u\big(h(\underline y);F'\big)$,
\item[$(iii)$] for every $\underline z\in \partial^- D_{F}$, we have $h_\#\circ F(\underline z)=F'\circ h_\#(\underline z)$.
\end{itemize}
\end{lemm}

We define $h$ on $\mathcal L_F$ via the following expression:
\[h:\; \underline x\in \mathcal L_F\mapsto F'^{n}\circ h_\#\circ F^{-n}(\underline x),\quad \mathrm{if}\; \underline x\in F^n(D_F),\; n\in\mathbb Z.\]

We notice that for every $\underline x\in \mathcal L_F$, we have:
\[ F' \circ h(\underline x)= h\circ F(\underline x).\]

This expression complements the above definition of $h$ on $\hat R_F$ and $A_F$ as hyperbolic continuation.

It is easy to see that the restriction of $h$ to $\mathcal L_F$ is continuous.
Moreover for every $\underline x\in \hat R_F$, $\underline y\in A_F$, the map $h$ sends $W^s(\underline x; F)\cap W^u(\underline y; F)$ into $W^s(h(\underline x); F')\cap W^u(h(\underline y); F')$. As moreover $\pi_0\circ h$ is $C^1$-close to $\pi_0$, the map $h$ is injective.

To prove that $h$ sends $M_F$ onto $M_{F'}$, we need to prove first the  global continuity of $h$. In order to do so, it remains only to show that the definition of $h$ on $\mathcal L_F$ and the definition of $h$ on $\hat R_F\cup A_F$ fit together continuously.

\paragraph{Proof of the continuity at $\hat R_F$}
Let $(\underline x^n)_{n\ge 0}$ be a sequence of points in $\mathcal L_F$ approaching to $\underline x\in \hat R_F$.
We want to show that $\big(h(\underline x^n)\big)_{n\ge 0}$ approaches $h(\underline x)$.
\
The set $\cup_{\underline x'\in W^s_\epsilon(\underline x;F)} W^u_\epsilon(\underline x',\epsilon)$ is a distinguished neighborhood of $\underline x$. Thus, for $n$ large, there exists $\underline x'^n\in W^s_\epsilon(\underline x;F)$ such that $\underline x^n$ belongs to $W^u_\epsilon(\underline x'^n;F)$. Actually for $n$ large, the point $\underline x^n$ is much closer to $\underline x'^n$ than $\epsilon$. Also $(x'_n)_n$ converges to $\underline x$.

As each $\underline x'^n$ is in $\hat R_F$, for $n $ sufficiently large, the point $h(\underline x^n)$ belongs to $W^u_\epsilon(h(\underline x'^n); F')$. By continuity of $h$ and of the holonomy of $\mathcal L_{F'}$, any limit point $\underline z$ of $\big(h(\underline x^n)\big)_{n\ge 0}$ belongs to $W^u_\epsilon\big(h(\underline x); F'\big)$.

We can do the same proof for the sequence $\big(F^k(\underline x^n)\big)_{n\ge 0}$, from which we get that any limit point of $\big(h\circ F^k(\underline x^n)\big)_{n\ge 0}$ belongs to $W^u_\epsilon\big(h(F ^k(\underline x'));F'\big)$. By using the equality $h\circ F^k(\underline x^n)=F'^k\big(h(\underline x^n)\big)$ and the continuity of $F'$, we note that the iterate $F'^k(\underline z)$ is a limit point of $\big(h\circ F^k(\underline x^n)\big)_n$. Thus $F'^k(\underline z)$ belongs to $W^u_\epsilon\big(h\circ F ^k(\underline x); F'\big)=W^u_\epsilon\big(F'^k(h(\underline x));F'\big)$ for every $k\ge 0$. By expansion along the unstable manifolds, the point $\underline z$ must be $h(\underline x)$.

\paragraph{Proof of the continuity at $A_F$} Let $(\underline x^n)_{n\ge 0}$ be a sequence of $\mathcal L$ approaching to $\underline x\in A_F$. We are going show that $(h(\underline x^n))_{n\ge 0}$ approaches $h(\underline x)$, by the same way as above, but this time we work on $M$.

Indeed, by taking a $\mathcal L^s_F$-distinguished neighborhood, we have that any limit point $\underline z$ of $(h(\underline x^n))_{n\ge 0}$ satisfies that $\pi_{0}(\underline z)$ belongs to $W^s_\epsilon(\pi_0\circ h(\underline x);f')$. The same holds for $\pi_0(F^{-k}(\underline z))=
\pi_{-k}(\underline z)$: it belongs to $W^s_\epsilon(\pi_{-k}\circ h(\underline x);f')$, for every $k\ge 0$. By contraction of the stable manifold, this means that $\pi_{-k}(\underline z)$ is equal to $\pi_{-k}\circ h(\underline x)$ for every $k\ge 0$. In other words, $\underline z$ is equal to  $h(\underline x)$.

\paragraph{Surjectivity of $h$}
The proof is not obvious since $W^u(\hat R_{F'})$ is not always connected and lands in the space $M_F$ which is not necessarily a manifold.

Let us show that the image of  $h$ contains a neighborhood of $\hat R_F$. This implies that the image contains a fundamental domain for $W^u(\hat R_{F'})$ and so by conjugacy that the image of $h$ contains $M_{F'}=W^u(\hat R_{F'})\sqcup A_{F'}$ by Proposition \ref{liminvpert}.

For every $\underline x\in \hat R_F$, the map $h$ sends the local unstable manifold $W^u_\epsilon(\underline x)$ into a subset of $W^u_{\hat \epsilon}\big (h(\underline x)\big)$ which contains $h(\underline x)$. As $h$ is a homeomorphism onto its image, its restriction to this manifold is a homeomorphism onto its image  which is a manifold of same dimension. Thus $h(W^u_\epsilon(\underline x))$ is an open neighborhood of $h(\underline x)$ in $W^u_{\hat \epsilon}\big(h(\underline x)\big)$. By compactness of $\hat R_F$, there exists $\eta>0$ such that for every $\underline x\in \hat R_F$, the open set $h\big(W^u_\epsilon(\underline x)\big)$ contains $W^u_\eta\big(h(\underline x)\big)$. This implies that the image of $h$ contains $W^u_\eta(\hat R_{F'})$ which is a neighborhood of $\hat R_{F'}$.
\subsection{Proof of Lemma $\ref{unlemmedeconju}$}
Let $V$ be a precompact, open neighborhood of $D_F$ in $W^u_{\hat \epsilon}(\hat R_{ F})\setminus \hat R_{ F}\subset \mathcal L_F$. We recall that $\hat \epsilon>\epsilon$. 
\begin{lemm}  There exists $I:\; V\supset D_F\rightarrow \mathcal L_{F'}$ a homeomorphism onto its image such that:
\begin{itemize}
\item For every $\underline z\in V$, the point $I(\underline z)$ belongs to $W^s(h(\underline x);F')\cap W^u  (h(\underline y);F')$ if $\underline z$ belongs to $W^s(\underline x;F)\cap W^u  (\underline y;F)$, with $\underline x\in \hat R_F$ and $\underline y\in A_F$,
\item   the map $i_0:=\pi_0\circ I$ is $C^1$-close to $\pi_0|V$ when $f'$ if close to $f$.
\end{itemize}
\end{lemm}
\begin{proof} Let $N\ge 0$ be such that $F^N(V)$ has its closure in $W^s_\epsilon(A_F)$. Let us first notice that the images by $\pi_0$ of $W^u_{\hat \epsilon} \big(h(\underline x); F'\big)$ and $W^s_\epsilon\big(h(\underline y); F'\big)$ depend continuously on $f'$, $\underline x$ and $\underline y$ for the $C^1$-topologies.

For $\underline z\in V$, let $L_{\underline z}$ be the set of pairs $(\underline x,\underline y)\in \hat R_F\times A_F$ such that
$\underline z$ belongs to $W^u_{\hat \epsilon}(\underline x; F)\cap F^{-N}\big(W^s_{ \epsilon}(\underline x; F)\big)$. Put:
\[\mathcal L_z:= \bigcup_{(\underline x,\underline y)\in L_{\underline z}}W^u_{\hat \epsilon}(\underline x; F)\cap F^{-N}\big(W^s_{ \epsilon}(\underline y; F)\big)\quad \mathrm{and}\quad
\mathcal L'_z:= \bigcup_{(\underline x,\underline y)\in L_{\underline z}}W^u_{\hat \epsilon}(h(\underline x); F')\cap F'^{-N}\big(W^s_{ \epsilon}(h(\underline y); F')\big)\]
We remark that $\mathcal L'_z$ and $\mathcal L_z$ are manifolds, and $\mathcal L_z$ contains the $\mathcal L_F|V$-leaf of $\underline z$.
  
To accomplish the proof of the lemma, we endow the lamination $\mathcal L_F| V$ immersed by $\pi_0$ with a tubular neighborhood, that is a family of $C^1$-disks $(D_{\underline z'})_{\underline z'\in \mathcal L_F}$ embedded into $M$ such that:
\begin{itemize}
\item $D_{\underline z}$ is transverse to $\pi_0 (\mathcal L_{\underline z})$ and satisfies $D_{\underline z}\pitchfork_{\pi_0} \mathcal L_{\underline z}=\{\underline z\}$,
\item the disks of each small $\mathcal L_F$-plaque form the leaves of a $C^1$-foliation of an open subset of $M$,
\item these foliations depend  $C^1$-continuously transversally to $\mathcal L_F$.
\end{itemize}
By \cite{berlam}, Prop. 1.5,  any $C^1$-immersed lamination has a tubular neighborhood.

For $f'$ sufficiently close to $f$, the submanifold $\pi_0(\mathcal L'_{\underline z})$ intersects $D_{\underline z}$ at a unique point $i_0(\underline z)$, for every $\underline z\in V$.  By transversality, the map $i_0:\; V\rightarrow M$ is of class $C^1$.

We put $I(\underline z):= \pi_0^{-1}\big(i_0(\underline z)\big)\cap W^u_{\hat \epsilon}\big(h(\underline x);F'\big)$, with $\underline z\in W^u_{\hat \epsilon}(\underline x;F)$. Such a map satisfies the required properties.
\end{proof}

Let $W$ be a small neighborhood of $\partial^- D_F$ such that the closures of $W$ and $F(W)$ are disjoint and included in $V$.

Let us modify $I$ to a map $h_\#$ which satisfies moreover that for every $\underline z\in W$, $h_\#\circ F(\underline z)= F'\circ h_\#(\underline z)$.

 We define $h_\# $ on $F(W)$ as equal to $I$ and on $W$ as equal to  $h_1:=F'^{-1}\circ I\circ F$.

Between,  $h_\#$ will be such that it respects the lamination $\mathcal L_{F'}$ and remains $C^1$ close  to $I$.

For this end, let us  define a map $h_2:\; V\rightarrow M$ equal to $i_0$ on $F(W)$ and to $\pi_0\circ h_1$ on $W$.

Take  a $C^1$-function $\rho$ equal to $1$ on $W$ with support in a small neighborhood $\hat W$ of $W$ (disjoint from  $F(\hat W)$) in $V$.
Let $\exp$ be the exponential map  associated to a Riemannian metric of $M$.

Put:
 \[h_2:\; \underline z\in V\mapsto \left\{ \begin{array}{cl} \exp_{i_0(\underline z)} \Big[ \rho(\underline z)\cdot \exp^{-1}_{i_0(\underline z)} \big(\pi_0\circ h_1(\underline z)\big) \Big]& \mathrm{if}\; \underline z\in \hat W,\\
  i_0(\underline z) & \mathrm{otherwise.}\end{array}\right.\]
  The map $h_2$ is of class $C^1$ as composition of $C^1$-maps. Moreover it is $C^1$-close to $\pi_0$ since $\pi_0 \circ h_1$ and $i_0$ are $C^1$-close to $\pi_0$. In particular, for $f'$ close to $f$, $h_2$ is an immersion of the lamination $\mathcal L_F|V$.

We notice that $h_2$ sends $\mathcal L_F$ plaques included in $W\cup F(W)$ into the $\pi_0$-image of $\mathcal L_{F'}$-plaques.

In order to construct the map $h_\#:D_F\rightarrow \mathcal L_{F'}$ from $h_2$, we take a tubular neighborhood $(D_{\underline z})_{\underline z\in V}$ of $\mathcal L_F$ (see the definition in the proof of the above lemma).

For $\underline z\in V$, the point $h_2(\underline z)$ is close to $\pi_0(\underline z)$ and so belongs to a unique disk $D_{\underline z'}$ with ${\underline z'}\in \mathcal L_{\underline z}$. Also $\pi_0(\mathcal L'_{\underline z})$ intersects $D_{\underline z'}$ at a unique point. Let $h_\#(\underline z)$ be the preimage of this point by $\pi_0|\mathcal L'_{\underline z}$.

%in a small $\mathcal L_F$-plaque containing $\underline z$. A small $\mathcal L_F'$-plaque containing $I(\underline z)$ intersects also  $D_{\underline z'}$ at a unique point. Let $h_\# (\underline z)$ be the intersection point.

We note that $h_\#$ sends each $\mathcal L_F$-plaque included in $V$ into a $\mathcal L_{F'}$-plaque. By smoothness of the holonomy between two transverse sections of a $C^1$-foliation, the map $\pi_0\circ h_\#$ is of class $C^1$. This concludes the proof of Lemma \ref{unlemmedeconju}.
\begin{flushright}
$\square$
\end{flushright}

 \bibliographystyle{alpha}
\bibliography{references}

\end{document}